\documentclass[11pt]{amsart}

\usepackage{amsmath,amssymb}
\usepackage{amsthm}
\usepackage{hyperref}
\usepackage[margin=1.0in]{geometry}

\numberwithin{equation}{section}

\newtheorem{prop}{Proposition}
\newtheorem{lemma}[prop]{Lemma}

\newtheorem{thm}[prop]{Theorem}
\newtheorem{cor}[prop]{Corollary}

\numberwithin{prop}{section}

\theoremstyle{definition}
\newtheorem{defn}[prop]{Definition}
\newtheorem{ex}[prop]{Example}
\newtheorem{rmk}[prop]{Remark}

\newcommand{\del}{\partial}
\newcommand{\delb}{\bar{\partial}}\newcommand{\dt}{\frac{\partial}{\partial t}}
\newcommand{\sdt}{\tfrac{\partial}{\partial t}}
\newcommand{\brs}[1]{\left| #1 \right|}

\newcommand{\gs}{\sigma}

\newcommand{\gw}{\omega}

\newcommand{\N}{\nabla}

\newcommand{\til}[1]{\widetilde{#1}}

\renewcommand{\bar}[1]{\overline{#1}}

\renewcommand{\i}{\sqrt{-1}}

\newcommand{\IP}[1]{\left<#1\right>}

\newcommand{\JJ}{\mathbf{J}}
\newcommand{\GG}{\mathbf{G}}

\DeclareMathOperator{\Rc}{Rc}

\DeclareMathOperator{\tr}{tr}

\DeclareMathOperator{\Img}{Im}

\DeclareMathOperator{\End}{End}

\begin{document}

\title[Deformation classes in generalized K\"ahler geometry]{Deformation classes in generalized K\"ahler geometry}

\begin{abstract} We introduce natural deformation classes of generalized K\"ahler structures using the Courant symmetry group.  We show that these yield natural extensions of the notions of K\"ahler class and K\"ahler cone to generalized K\"ahler geometry.  Lastly we show that the generalized K\"ahler-Ricci flow preserves this generalized K\"ahler cone, and the underlying real Poisson tensor.
\end{abstract}

\date{May 5th, 2020}

\author{Matthew Gibson}
\address{Rowland Hall\\
         University of California\\
         Irvine, CA 92617}
\email{\href{mailto:gibsonmd@uci.edu}{gibsonmd@uci.edu}}

\author{Jeffrey Streets}
\address{Rowland Hall\\
         University of California\\
         Irvine, CA 92617}
\email{\href{mailto:jstreets@uci.edu}{jstreets@uci.edu}}

\maketitle

\section{Introduction}

A rudimentary notion of K\"ahler geometry is that of the K\"ahler class: given $(M^{2n}, \gw, J)$ a K\"ahler manifold, the K\"ahler form $\gw$ is closed, and $[\gw] \in H^{1,1}_{\mathbb R}$ is the associated K\"ahler class.  Fixing the complex structure $J$, the space of all K\"ahler classes defines an open cone in $H^{1,1}_{\mathbb R}$ (the K\"ahler cone), and the fundamental result of Demailly-Paun \cite{DemaillyPaun} gives a characterization of this cone in terms of pairing against complex subvarieties.  The space of K\"ahler metrics within a given K\"ahler class is an open infinite dimensional cone in $C^{\infty}(M)$ using the $\i\del\delb$-lemma.  Thus the basic structure of the space of K\"ahler metrics compatible with a fixed complex structure $J$ is fairly well understood.  If we instead ask for the space of all K\"ahler pairs $(g, J)$ on a given smooth manifold $M$, the question becomes decidedly more delicate.  The global structure can be quite wild, with disconnected components of arbitrarily large dimension (cf. \cite{Brieskorn}).

Understanding the space of \emph{generalized K\"ahler} structures on a given manifold $M$ becomes even more delicate.  Originally discovered by Gates-Hull-Rocek \cite{GHR}, a generalized K\"ahler structure on a smooth manifold $M$ is a triple $(g, I, J)$ consisting of a Riemannian metric $g$ compatible with two integrable complex structures $I, J$ further satisfying
\begin{align*}
d^c_I \gw_I = H = - d^c_J \gw_J, \qquad d H = 0.
\end{align*}
Later, Gualtieri \cite{GualtieriThesis} gave a natural description of this geometry using the language of Hitchin's generalized complex structures \cite{HitchinGCY}, in particular in terms of a pair of generalized complex structures $(\JJ_1, \JJ_2)$ satisfying some natural conditions (cf. \S \ref{ss:bck}).  A fundamental question is to understand the degrees of freedom, moduli, and topology of the space of generalized K\"ahler structures on a given smooth manifold.

Whereas in the K\"ahler setting we can roughly speaking divide the problem of understanding the space of K\"ahler metrics into the space of possible complex structures and then to consider the space of compatible K\"ahler metrics, in the generalized K\"ahler setting such a decomposition is not really possible.  Indeed, in many settings, given two complex structures $I, J$, there is at most one compatible metric which defines a generalized K\"ahler structure.  Nonetheless, many different classes of deformations of generalized K\"ahler structure have been constructed.  Joyce gave the first examples of nontrivial (i.e. non-K\"ahler) generalized K\"ahler structures by deforming away from hyperK\"ahler structures (cf. \cite{AGG}), specifically using an action of diffeomorphisms which are Hamiltonian with respect to an associated holomorphic symplectic structure.  Later Hitchin produced nontrivial generalized K\"ahler structures on del Pezzo surfaces, with a choice of holomorphic Poisson structure playing a key role \cite{HitchindelPezzo}.  Also, Goto \cite{Goto} has extended the stability result of Kodaira-Spencer to the generalized K\"ahler setting, with the restriction that one of the generalized complex structures be defined by a pure spinor.

Our main purpose in this work is to describe a class of deformations which generalizes and unifies different notions of ``K\"ahler class" arising in the different flavors of generalized K\"ahler geometry.  It is well-known that two-forms ($B$-fields) can act on generalized complex structures by conjugation, with the integrability condition being preserved if and only if $B$ is closed.  Our deformations exploit a different, and moreover infinitessimal, action of $B$-fields.  In particular, we will say (cf. Definition \ref{d:GKvariation}) that a one-parameter family of generalized K\"ahler structures is a \emph{canonical deformation} if there exists a one parameter family $K_t \in \Lambda^2$ such that for all times $t$ where defined, one has
\begin{align*}
\dt \JJ_1 = [\JJ_1, e^K \JJ_1], \qquad \dt \JJ_2 = [\JJ_2, e^K \JJ_2].
\end{align*}

Equivalence classes of canonical deformations lead to natural definitions of generalized K\"ahler class and generalized K\"ahler cone (cf. \S \ref{s:VGK}).  These definitions make nonobvious departures from the classical idea of K\"ahler class and K\"ahler cone.  The first is the use of infinitessimal deformations  as opposed to `large' deformations.  Whereas any two metrics in the same K\"ahler class admit an explicit relationship using the $\i\del\delb$-lemma, we can no longer expect such an explicit relationship in general.  For instance, as described above Joyce's construction of nontrivial GK structure uses diffeomorphisms which are Hamiltonian with respect to the associated holomorphic symplectic structure, and in general these cannot be described by a single potential function.  Instead, as is typical of Hamiltonian diffeomorphisms, we expect to be able to explicitly describe their infinitessimal deformations.  Moreover, given that in the K\"ahler setting, deformations in the K\"ahler class involve freezing the complex structure and varying the K\"ahler form, it is natural to imagine that one should deform while fixing either $\JJ_1, \JJ_2$.  Nonetheless through careful consideration of natural variational classes of different flavors of generalized K\"ahler metrics it emerges that varying $\JJ_1$ and $\JJ_2$ simultaneously will correctly capture various existing notions of K\"ahler class in GK geometry.

A fundamental first step in unpacking this definition is to derive the algebraic and differential conditions which are imposed on $K$ to preserve the compatibility and integrability conditions for the pair $(\JJ_1, \JJ_2)$.  Through careful computations, it turns out that the answer is pleasingly simple:
\begin{thm} \label{t:maindeformationthm} Given $M$ a smooth manifold, suppose $(\JJ_1, \JJ_2)$ is a generalized K\"ahler structure.  Suppose $(\JJ_1^t, \JJ_2^t)$ is a one-parameter family of generalized almost complex structures such that
\begin{align*}
\dt \JJ_1 = [\JJ_1, e^K \JJ_1], \qquad \dt \JJ_2 = [\JJ_2, e^K \JJ_2],
\end{align*}
for some one parameter family $K_t \in \Lambda^2$.  Then $(\JJ_1^t, \JJ_2^t)$ is a one-parameter family of generalized K\"ahler structures if and only if for all $t$ one has 
\begin{enumerate}
\item $K_t \in \Lambda^{1,1}_{J^t}$,
\item $d K_t = 0$,
\item $\IP{-\JJ_1^t \JJ_2^t \cdot, \cdot } > 0$.
\end{enumerate}
where $J^t$ is determined via the Gualtieri map, and $\IP{,}$ denotes the symmetric neutral inner product on $T \oplus T^*$ (cf. \S \ref{ss:GKbck}).
\end{thm}

In particular, this theorem exhibits that the canonical deformations are, as is true in the K\"ahler setting, determined infinitessimally by a closed form which is $(1,1)$ with respect to $J$.  We emphasize here that the condition that $dK = 0$ does \emph{not} follow from the known fact that the conjugation action of $B$-fields on generalized complex structures preserves integrability if and only if $dB = 0$.  For instance, if we consider our infinitessimal action on a single generalized complex structure, the condition to preserve integrability is \emph{strictly weaker} than $dK = 0$ (cf. Proposition \ref{p:integrability}).  It is only in the context of preserving the integrability conditions of generalized \emph{K\"ahler} geometry that one derives $dK = 0$.

Despite the simplicity of the conditions of Theorem \ref{t:maindeformationthm} and the apparent simplicity of canonical deformations from the point of view of generalized geometry, the deformations induced on the classical bihermitian data $(g,I, J)$ are delicate.  Remarkably, these canonical deformations unify all previously known instances of ``K\"ahler class" in generalized geometry, specifically the classical notion of K\"ahler class, the modified K\"ahler classes implicit in Apostolov-Gualtieri (\cite{ApostolovGualtieri} Proposition 5, cf. also \cite{GHR}) in the commuting GK case, as well as Joyce's Hamiltonian deformation construction in the nondegenerate case.  We state this for emphasis (cf. \S \ref{ss:examples} for notation):

\begin{prop} \label{p:examplesprop} The following hold:
\begin{enumerate}
\item Given $(M^{2n}, g, J)$ a K\"ahler manifold, and $u \in C^{\infty}(M)$ such that $\gw + \i \del \delb u > 0$, the one-parameter family
\begin{align*}
(\gw_I)_t = \gw + t d I d u, \qquad I_t = J, \qquad J_t = J
\end{align*}
arises as a canonical deformation of generalized K\"ahler structures for $0 \leq t \leq 1$ defined by
\begin{align*}
K_t = d J d u.
\end{align*}
\item Given $(M^{2n}, g, I, J)$ a generalized K\"ahler manifold such that $[I, J] = 0$, and $u \in C^{\infty}(M)$ such that $\gw_I + \i \left( \del_+ \delb_+ - \del_- \delb_-  \right)u > 0$, the one-parameter family
\begin{align*}
(\gw_I)_t = \gw_I + t \i \left( \del_+ \delb_+ - \del_- \delb_-  \right) u, \qquad I_t = I, \qquad J_t = J
\end{align*}
arises as a canonical deformation of generalized K\"ahler structures for $0 \leq t \leq 1$ defined by
\begin{align*}
K_t = d J d u.
\end{align*}
\item Let $(M^{2n}, g, I, J)$ be a generalized K\"ahler manifold such that the Poisson structure $\gs = \tfrac{1}{2}[I, J] g^{-1}$ is nondegenerate, with $\Omega = \gs^{-1}$.  Given $u_t \in C^{\infty}(M)$ a family of smooth functions, let $\phi_t$ denote the one-parameter family of $\Omega$-Hamiltonian diffeomorphisms generated by $u_t$.  Then, for all $t$ such that $- \Img \pi_{\Lambda_{I}^{1,1}} \Omega_J > 0$, the one-parameter family of generalized K\"ahler structures determined by
\begin{align*}
\Omega_t = \Omega, \qquad I_t = I, \qquad J_t = \phi_t^* J
\end{align*}
defines a canonical deformation of generalized K\"ahler structures determined by
\begin{align*}
K_t = d J_t d u_t.
\end{align*}
\end{enumerate}
\end{prop}

As a final point to contextualize these deformations, we recall that various interesting deformation classes of generalized K\"ahler structure have been produced using holomorphic Poisson structures.  Given a generalized K\"ahler structure $(g, I, J)$, there is a Poisson tensor
\begin{align*}
\gs = \tfrac{1}{2} [I, J] g^{-1}
\end{align*}
which is the real part of a holomorphic Poisson tensor with respect to both $I$ and $J$.  By choosing an appropriate deformation of $\gs$, Hitchin \cite{HitchindelPezzo} produced deformations of K\"ahler metrics on del Pezzo surfaces to strictly generalized K\"ahler structures.  Also the deformation theory of Goto \cite{Goto} changes this underlying Poisson tensor.  As it turns out our deformations fix $\gs$ and $I$, so occur against a fixed background of a holomorphic Poisson structure.

\begin{cor} \label{c:Poissoncor} Given $M$ a smooth manifold, suppose $(\JJ_1^t, \JJ_2^t)$ is a canonical deformation of generalized K\"ahler structures.  Then for all $t$,
\begin{align*}
I^t \equiv I^0, \qquad \gs^t \equiv \gs^0
\end{align*}
\end{cor}

As an application, we are able to express the generalized K\"ahler-Ricci flow in a simple way using canonical deformations.  The equation is an extension of K\"ahler-Ricci flow to the setting of generalized K\"ahler geometry, introduced by the second author and Tian \cite{GKRF}.  Recently this flow has been used to study the global topology of the (nonlinear) space of generalized K\"ahler structures in certain settings \cite{ASNDGKCY}.  To describe this flow, fix $(g, I, J)$ a generalized K\"ahler structure.  Associated to the Hermitian structure $(g, I)$ is the Bismut connection
\begin{align*}
\N^I = D + \tfrac{1}{2} H g^{-1},
\end{align*}
where $H = d^c_I \gw_I$, and $D$ denotes the Levi-Civita connection.  This is a Hermitian connection, and if $\Omega_I$ denotes its curvature, we obtain a representative of the first Chern class via contraction, called the Bismut-Ricci tensor:
\begin{align*}
\rho_I =&\ \tfrac{1}{2} \tr \Omega_I I.
\end{align*}
From the Bianchi identity we know that $d \rho_I = 0$, but it is not in general true that $\rho_I \in \Lambda^{1,1}_{I}$, and we will let $\rho_I^{1,1}$ denote its $(1,1)$ projection.  Furthermore, associated to $(g, I)$ we obtain the \emph{$I$-Lee form}, defined by
\begin{align*}
\theta_I(X) = d^* \gw_I(I X).
\end{align*}
Similarly we obtain the Lee form $\theta_J$ associated to $(g, J)$.  With this background in place, we can describe the generalized K\"ahler-Ricci flow in the $I$-fixed gauge simply by
\begin{gather} \label{f:GKRFBH}
\begin{split}
\dt \gw_I =&\ - \rho_I^{1,1}, \qquad \dt J = L_{\tfrac{1}{2} \left(\theta_J^{\sharp} - \theta_I^{\sharp}\right)} J.
\end{split}
\end{gather}
The evolution of the complex structure $J$ is derived in \cite{GKRF}, arising from delicate gauge manipulations and curvature identities.  On the other hand it has been shown in several special cases (cf. \S \ref{s:GKRF} below) that the generalized K\"ahler-Ricci is driven entirely by $\rho_I$.  Using our description of canonical deformations, and a further subtle curvature identity for generalized K\"ahler manifolds (Proposition \ref{p:sigchern2}), we confirm that this is true in full generality, and give a very simple description of generalized K\"ahler-Ricci flow in terms of the associated generalized complex structures.

\begin{thm} \label{t:GKflowthm} Let $(M^{2n}, g_t, I, J_t)$ be a solution of generalized K\"ahler-Ricci flow in the $I$-fixed gauge.  The one parameter family of pairs of associated generalized complex structures $(\JJ_1^t, \JJ_2^t)$ evolve by
\begin{gather} \label{f:GKRFGG}
\begin{split}
\dt \JJ_1 =&\ [\JJ_1, e^{\rho_I} \JJ_1], \qquad \dt \JJ_2 = [\JJ_2, e^{\rho_I} \JJ_2].
\end{split}
\end{gather}
In other words, the generalized K\"ahler-Ricci flow is the canonical deformation driven by the $I$-Bismut-Ricci tensor.
\end{thm}

Immediately following from Theorem \ref{t:GKflowthm} and Corollary \ref{c:Poissoncor} is that fact that generalized K\"ahler-Ricci flow preserves the underlying real Poisson tensor $\gs$, and moreover preserves the generalized K\"ahler cone associated to the initial data.

\begin{cor} Let $(M^{2n}, g_t, I, (J)_t)$ be a solution of generalized K\"ahler-Ricci flow in the $I$-fixed gauge. 
The associated one-parameter families of generalized complex structures $(\JJ_1, \JJ_2)$ lies in the generalized K\"ahler cone associated to the initial data.  In particular, for all $t$ such that the flow is defined, 
\begin{align*}
\gs^t \equiv \gs^0,
\end{align*}
in other words, the real Poisson tensor $\gs$ is fixed along the flow.  
\end{cor}

\section{Formal deformations of generalized complex structure} \label{s:GCsec}
\subsection{Background} \label{ss:bck}

Given $M$ a smooth manifold, the generalized tangent bundle is given by $T \oplus T^*$.  This bundle comes equipped with a family of natural brackets determined by a closed three-form $H$.  In particular, given $H \in \Lambda^3 T^*$, $d H = 0$, define the twisted Courant bracket $[,]$ for sections of $T \oplus T^*$ via
\begin{align} \label{f:Courant}
[X + \xi, Y + \eta] =&\ [X,Y] + L_X \eta - L_Y \xi + \tfrac{1}{2} d \left( \xi(Y) - \eta(X) \right) + i_Y i_X H.
\end{align}
A generalized complex structure $\JJ$ is then an almost complex structure on $T \oplus T^*$, whose $\i$-eigenbundle, denoted by $L$, is integrable with respect to the the twisted Courant bracket.  This  condition is naturally captured by a corresponding version of the Nijenhuis tensor, where for a given almost complex structure we associate the natural projection maps $\pi_{0,1}, \pi_{1,0}$ and then for $\vec{x}, \vec{y} \in T \oplus T^*$ we have
\begin{align} \label{f:Nij}
N_{\JJ}(\vec{x}, \vec{y}) = \pi_{0,1} [\pi_{1,0}(\vec{x}), \pi_{1,0}(\vec{y})].
\end{align}
Direct computations show that this is tensorial, and vanishes if and only if the associated almost generalized complex structure is integrable.  See \cite{GualtieriThesis} for further discussion.

\subsection{Variations of generalized complex structure} \label{ss:GCvar}

To begin we define an action of $B$-fields on generalized complex structures.

\begin{defn} \label{d:vardef}  Given a smooth manifold $M$ and $K \in \Lambda^2$, define
\begin{align*}
\Phi_K & : \End(T \oplus T^*)\to \End(T \oplus T^*),\\
\Phi_K& (\JJ) = [\JJ, e^K \JJ],
\end{align*}
where
\begin{align*}
e^K = \left(
\begin{matrix}
1 & 0\\
K & 1
\end{matrix} \right) \in \End(T \oplus T^*).
\end{align*}
\end{defn}

For a given $\JJ$, we intend to use $\Phi_K(\JJ)$ as a tangent vector to a one-parameter variation of $\JJ$ through generalized complex structures.  We first note that variations of this kind will indeed preserve the space of generalized almost complex structures.

\begin{lemma} \label{l:variation}  Let $\JJ_t$ be a one-parameter family of endomorphisms of $T\oplus T^*$ such that $\JJ_0$ is an almost generalized complex structure and
\begin{align*}
\dt \JJ_t =&\ [\JJ_t, e^{K_t} \JJ_t].
\end{align*}
Then $\JJ_t^2 = -1$ for each $t$, i.e. $\JJ_t$ is a family of generalized almost complex structures.
\end{lemma}
\begin{proof}
Differentiating the expression $\JJ_t^2$ yields
\begin{align*}
\dt \JJ_t^2 = \left(\dt \JJ \right) \JJ + \JJ \left( \dt \JJ \right) = \left(\JJ e^K\JJ +e^K\right)\JJ +\JJ \left(\JJ e^K\JJ +e^K\right)=0.
\end{align*}
Thus $\dt \JJ_t^2 =0$, and since $\JJ_0^2 =-1$ the lemma follows.
\end{proof}

Next we can characterize the condition for these deformations to preserve integrability of $\JJ_t$.

\begin{prop}\label{p:integrability}
Let $\JJ_t$ be a one-parameter family of generalized almost complex structures such that
\begin{align*}
\left. \dt \JJ_t \right|_{t=0} = \Phi_K(\JJ_0).
\end{align*}
Then for $\vec{x}, \vec{y} \in T \oplus T^*$, 
\begin{align*}
\left. \dt \right|_{t=0} N_{\JJ_t}(\vec{x}, \vec{y}) =&\ \i \pi_{0,1} ( dK \left( \pi_T \pi_{1,0}(\vec{y}), \pi_T \pi_{1,0}(\vec{x}), \cdot\right))\\
&\ - \i N_{\JJ_0}(\vec{x}, \vec{y}) + \JJ_0 e^K N_{\JJ_0}(\vec{x}, \vec{y}) + N_{\JJ_0}( e^K \JJ_0 (\vec{x}), \vec{y}) + N_{\JJ_0}(\vec{x}, e^K \JJ_0 \vec{y}).
\end{align*}
\end{prop}

\begin{proof} For notational simplicity we set $\JJ = \JJ_0$.  We differentiate the formula (\ref{f:Nij}) at $t=0$, using the explicit formulae for the projection maps, to obtain
\begin{gather} \label{int}
\begin{split}
\sdt  & (\pi^t_{0,1}\left[\pi^t_{1,0}(\vec{x}), \pi^t_{1,0}(\vec{y})\right])\\
=&\ \sdt(\pi^t_{0,1})\left[\pi^t_{1,0}(\vec{x}), \pi^t_{1,0}(\vec{y})\right] + \pi^t_{0,1}\left[\sdt\pi^t_{1,0}(\vec{x}), \pi^t_{1,0}(\vec{y})\right] + \pi^t_{0,1}\left[\pi^t_{1,0}(\vec{x}), \sdt\pi^t_{1,0}(\vec{y})\right]\\
=&\ \tfrac{\sqrt{-1}}{2}\left\{ \Phi_K(\JJ)[\pi_{1,0}(\vec{x}), \pi_{1,0}(\vec{y})]-\pi_{0,1}[\Phi_K(\JJ)(\vec{x}), \pi_{1,0}(\vec{y})]-\pi_{0,1}[\pi_{1,0}(\vec{x}),\Phi_K(\JJ)(\vec{y})] \right\}.
\end{split}
\end{gather}

Note that for any $\vec{z}\in T \oplus T^*$ we may write $\vec{z} =2\pi_{1,0}(\vec{z}) + \sqrt{-1}\JJ(\vec{z})$ which leads to  
\begin{equation*}
\begin{aligned}
\Phi_K(\JJ)(\vec{z}) &= (\JJ e^K\JJ + e^K)(\vec{z})\\
 &=\JJ e^K\JJ(\vec{z}) + e^K(2\pi_{1,0}(\vec{z}) +\sqrt{-1}\JJ(\vec{z}))\\
 &= 2e^K\pi_{1,0}(\vec{z}) + \sqrt{-1}\left( e^K\JJ\vec{z} -\sqrt{-1}\JJ(e^K\JJ\vec{z}) \right)\\
 &= 2e^K\pi_{1,0}(\vec{z}) + 2\sqrt{-1}\pi_{1,0}\left(e^K\JJ(\vec{z})\right).
\end{aligned}
\end{equation*}

This observation allows us to rewrite the final line of \ref{int} as
\begin{align*}
... =&\ \sqrt{-1}e^K\pi_{1,0}[\pi_{1,0}(\vec{x}), \pi_{1,0}(\vec{y})]-\pi_{1,0}e^K\JJ[\pi_{1,0}(\vec{x}), \pi_{1,0}(\vec{y})]\\
&\ -\sqrt{-1}\pi_{0,1}[e^K\pi_{1,0}(\vec{x})+\sqrt{-1}\pi_{1,0}e^K\JJ(\vec{x}), \pi_{1,0}(\vec{y})]\\
&\ - \sqrt{-1}[\pi_{1,0}(\vec{x}), e^K\pi_{1,0}(\vec{y})+\sqrt{-1}\pi_{1,0}e^K\JJ(\vec{y})]\\
=&\ \sqrt{-1}e^K\pi_{1,0}[\pi_{1,0}(\vec{x}), \pi_{1,0}(\vec{y})]-\pi_{1,0}e^K\JJ[\pi_{1,0}(\vec{x}), \pi_{1,0}(\vec{y})]\\
&\ -\sqrt{-1}\pi_{0,1}[e^K\pi_{1,0}(\vec{x}), \pi_{1,0}(\vec{y})] -\sqrt{-1}\pi_{0,1}[\pi_{1,0}(\vec{x}), e^K\pi_{1,0}(\vec{y})]\\
&\ +\pi_{0,1}\left([\pi_{1,0} e^K\JJ (\vec{x}),\pi_{1,0}(\vec{y})]+[\pi_{1,0}(\vec{x}), \pi_{1,0}(e^K\JJ\vec{y})]\right).
\end{align*}

Focusing on the first two terms, letting $\vec{z}=[\pi_{1,0}(\vec{x}), \pi_{1,0}(\vec{y})]$,
\begin{align*}
&\sqrt{-1}e^K\pi_{1,0}\vec{z}-\pi_{1,0}e^K\JJ\vec{z}\\
&=\frac{\sqrt{-1}}{2}e^K(\vec{z} -\sqrt{-1}\JJ\vec{z}) - \tfrac{1}{2}(e^K\JJ\vec{z} -\sqrt{-1}\JJ e^K\JJ\vec{z})\\
&= \frac{\sqrt{-1}}{2}(e^K\vec{z}+\JJ e^K\JJ\vec{z})\\
&= \sqrt{-1}\pi_{0,1}e^K\vec{z} + \tfrac{1}{2}\JJ e^K\vec{z} +\frac{\sqrt{-1}}{2}\JJ e^K\JJ\vec{z}\\
&=\sqrt{-1}\pi_{0,1}(e^K\vec{z}) + \tfrac{1}{2}\JJ e^K(\vec{z} + \sqrt{-1}\JJ\vec{z})\\
&= \sqrt{-1}\pi_{0,1}(e^K\vec{z}) + \JJ e^K\pi_{0,1}(\vec{z}).
\end{align*}
Equation \ref{int} now simplifies further to 
\begin{gather}\label{simplification}
\begin{split}
... =&\ \sqrt{-1}\pi_{0,1}\left(e^K[\pi_{1,0}(\vec{x}), \pi_{1,0}(\vec{y})] -[e^K\pi_{1,0}(\vec{x}), \pi_{1,0}(\vec{y})]-[\pi_{1,0}(\vec{x}), e^K\pi_{1,0}(\vec{y})]\right)\\
&\ + \JJ e^K \pi_{0,1} [\pi_{1,0}(\vec{x}), \pi_{1,0}(\vec{y})] + \pi_{0,1} \left( [\pi_{1,0} e^K\JJ (\vec{x}),\pi_{1,0}(\vec{y})]+[\pi_{1,0}(\vec{x}), \pi_{1,0}(e^K\JJ\vec{y})]\right)\\
=&\ \sqrt{-1}\pi_{0,1}\left(e^K[\pi_{1,0}(\vec{x}), \pi_{1,0}(\vec{y})] -[e^K\pi_{1,0}(\vec{x}), \pi_{1,0}(\vec{y})]-[\pi_{1,0}(\vec{x}), e^K\pi_{1,0}(\vec{y})]\right)\\
&\ + \JJ e^K N(\vec{x}, \vec{y}) + N( e^K \JJ (\vec{x}), \vec{y}) + N(\vec{x}, e^K \JJ \vec{y}).
\end{split}
\end{gather}
The Courant bracket satisfies $e^K[X + \xi, Y + \eta] = [e^K(X + \xi), e^K (Y + \eta)] + \iota_X\iota_YdK$.  Using this together with the fact that the Courant bracket involving a section with no tangent vector component vanishes we see that the first term of  (\ref{simplification}) becomes
\begin{align*}
\i \pi_{0,1} & \left( [e^K\pi_{1,0}(\vec{x}), e^K\pi_{1,0}(\vec{y})]-[e^K\pi_{1,0}(\vec{x}), \pi_{1,0}(\vec{y})] -[\pi_{1,0}(\vec{x}), e^K\pi_{1,0}(\vec{y})] + dK \left( \pi_T \pi_{1,0}(\vec{y}), \pi_T \pi_{1,0}(\vec{x}), \cdot\right) \right)\\
=&\ \i \pi_{0,1} \left( - [\pi_{1,0} (\vec{x}), \pi_{1,0}(\vec{y})] +  dK \left( \pi_T \pi_{1,0}(\vec{y}), \pi_T \pi_{1,0}(\vec{x}), \cdot\right) \right)\\
=&\ - \i  N(\vec{x}, \vec{y}) + \i \pi_{0,1} ( dK \left( \pi_T \pi_{1,0}(\vec{y}), \pi_T \pi_{1,0}(\vec{x}), \cdot\right)).
\end{align*}
Collecting these computations gives the result.
\end{proof}

\begin{cor} \label{c:integrability}
Let $\JJ_t$ be a one-parameter family of generalized almost complex structures such that $\JJ_0$ is integrable and for all $t$ one has
\begin{align*}
\dt \JJ = \Phi_{K_t}(\JJ),
\end{align*}
where furthermore for all $\vec{x}, \vec{y} \in T \oplus T^*$ one has
\begin{align} \label{f:JMC}
\pi_{0,1}^t dK_t \left( \pi_T \pi^t_{1,0}(\vec{y}), \pi_T \pi^t_{1,0}(\vec{x}), \cdot\right) =&\ 0.
\end{align}
Then $\JJ_t$ are integrable for each $t$.
\end{cor}
\begin{proof} Choosing any Hermitian metric on $(T \oplus T^*) \otimes \mathbb C$, using Proposition \ref{p:integrability} and the hypothesis (\ref{f:JMC}) one directly derives for all $t$
\begin{align*}
\dt \brs{N_{\JJ_t}}^2 \leq&\ C(K,\JJ) \brs{N_{\JJ_t}}^2.
\end{align*}
Since $N_{\JJ_0} = 0$ the result follows from Gronwall's inequality.
\end{proof}

\begin{rmk} \label{integrability2}
We may also formulate integrability of generalized complex structures in terms of its $-\sqrt{-1}$-eigenbundle, $\bar{L}_t$, being Courant integrable. Replicating the above arguments with the roles of $\pi_{1,0}$ and $\pi_{0,1}$ reversed we obtain that the relevant condition on $K_t$ is 
\begin{align*}
0 = \pi_{0,1}^t dK_t \left( \pi_T \pi_{0,1}(\vec{y}), \pi_T \pi_{0,1}(\vec{x}), \cdot\right).
\end{align*}
\end{rmk}

\section{Variations of generalized K\"ahler structure} \label{s:VGK}

Having defined certain variations of generalized complex structure, we now extend this to defining variations of generalized K\"ahler structure.  A naive guess would be that we should simply take a variation of one of the underlying generalized complex structures and seek the further integrability conditions.  However, for reasons to be illuminated by the examples below, it is much more natural to vary \emph{both} generalized complex structures by a single $B$-field as described in \S \ref{ss:GCvar}.

\subsection{Background} \label{ss:GKbck}

A generalized K\"ahler structure is a pair of commuting generalized complex structures $\JJ_1, \JJ_2$ such that $\GG = - \JJ_1 \JJ_2$ defines a generalized metric, i.e. $\IP{\GG \cdot, \cdot}$ is a positive definite inner product on $T \oplus T^*$, where $\IP{,}$ denotes the symmetric neutral inner product on $T \oplus T^*$, i.e.
\begin{align*}
\IP{X + \xi, Y + \eta} =&\ \tfrac{1}{2} \left( \xi(Y) + \eta(X) \right).
\end{align*}
A fundamental theorem of Gualtieri (\cite{GualtieriThesis} Chapter 6) says that a generalized K\"ahler structure $(\JJ_1, \JJ_2)$ as defined here corresponds to a bihermitian structure $(g, I, J, b)$, with K\"ahler forms $\gw_I, \gw_J$, as described in the introduction.  The explicit relationship is given by 
\begin{equation}\label{f:Gualtierimap}
\JJ_{1/2}=\tfrac{1}{2}e^b\left( \begin{array}{cc}
I\pm J & -(\omega_I^{-1} \mp \omega_J^{-1}) \\[2pt]
\omega_I \mp \omega_J & -(I^*\pm J^*)
\end{array} \right)e^{-b}.
\end{equation}
We recall that a generalized K\"ahler structure induces a fourfold decomposition of the complexified generalized tangent bundle.  Specifically, letting $L_i$ and $\bar{L}_i$ denote the $\pm\sqrt{-1}$-eigenbundles of $\JJ_i$ respectively, we have the following decomposition: $$(T\oplus T^*)\otimes\mathbb{C} =L_1^+\oplus L_1^-\oplus \bar{L_1^-}\oplus \bar{L_1^+}:= (L_1\cap L_2)\oplus (L_1\cap \bar{L}_2)\oplus (\bar{L_1}\cap L_2)\oplus (\bar{L}_1\cap \bar{L}_2).$$

\subsection{Definitions}

\begin{defn} \label{d:GKvariation}  A one-parameter family of generalized K\"ahler structures $(\JJ_1^t, \JJ_2^t)$ is a \emph{canonical family} if for all $t$, there exists $K_t \in \Lambda^2$ such that
\begin{align*}
\dt \JJ_1^t =&\ \Phi_{K_t}(\JJ^t_1), \qquad \dt \JJ_2^t = \Phi_{K_t}(\JJ^t_2).
\end{align*}
Given $\til{\JJ}_1, \til{\JJ}_2$ another generalized K\"ahler structure, we define an equivalence relation where
\begin{align*}
(\til{\JJ}_1, \til{\JJ}_2) \sim (\JJ_1, \JJ_2)
\end{align*}
if and only if there exists a canonical family $(\JJ_1^t, \JJ_2^t)$, $t \in [0,1]$, such that $(\JJ_1^0, \JJ_2^0) = (\JJ_1, \JJ_2)$, $(\JJ_1^1, {\JJ}_2^1) = (\til{\JJ}_1, \til{\JJ}_2)$.  Furthermore, the \emph{generalized K\"ahler cone associated to  $(\JJ_1, \JJ_2)$} is
\begin{align*}
\mathcal{GK} (\JJ_1, \JJ_2) = \{ (\til{\JJ}_1, \til{\JJ}_2) \mbox{ generalized K\"ahler } |\ (\til{\JJ}_1, \til{\JJ}_2) \sim (\JJ_1, \JJ_2) \}.
\end{align*}
\end{defn}

\subsection{Compatibility Condition}

We first address the condition required for a canonical deformation to preserve the algebraic compatibility condition of generalized K\"ahler structures.  We first prove a formal lemma reducing this to an algebraic condition on $K$, then analyze this explicitly using the Gualtieri map.

\begin{lemma} \label{l:GKvariationlemma}  Let $\JJ_1^t, \JJ_2^t$ be one-parameter families of generalized almost complex structures, with $[\JJ_1^0, \JJ_2^0] = 0$, which satisfy
\begin{align*}
\dt \JJ_1^t = \Phi_{K_t}(\JJ_1), \qquad  \dt \JJ_2^t = \Phi_{K_t}(\JJ_2).
\end{align*}
Then $[\JJ_1^t, \JJ_2^t] = 0$ for all $t$ if and only if
\begin{gather*}
\begin{split}
 [\Phi_K(\JJ_1), \JJ_2] = [\Phi_K(\JJ_2), \JJ_1]
\end{split}
\end{gather*}
for all $t$.
\end{lemma}

\begin{proof}
Differentiating $[\JJ_1^t, \JJ_2^t]$ at any time $t$ shows 
\begin{align*}
\dt [\JJ_1^t, \JJ_2^t] &= \Phi_K(\JJ_1)\JJ_2 + \JJ_1\Phi_K(\JJ_2) - \Phi_K(\JJ_2)\JJ_1 - \JJ_2\Phi_K(\JJ_1)\\
& =[\Phi_K(\JJ_1), \JJ_2] -[\Phi_K(\JJ_2), \JJ_1].
\end{align*}
Since $[\JJ_1^0, \JJ_2^0] = 0$, the result follows.
\end{proof}

Here we reformulate the compatibility condition of Lemma \ref{l:GKvariationlemma} by expanding the necessary equation in terms of the Gualtieri map and analyzing the result, which simplifies dramatically.
\begin{prop} \label{p:GKcompatibilityprop} Given $M$ a smooth manifold and $(\JJ_1, \JJ_2)$ a generalized K\"ahler structure, for $K \in \Lambda^2$ one has
\begin{align*}
[\Phi_K(\JJ_1), \JJ_2] = [\Phi_K(\JJ_2), \JJ_1]
\end{align*}
if and only if
 \begin{align*}
 K \in \Lambda^{1,1}_{J}.
 \end{align*}
\end{prop}
\begin{proof}
Let $\Upsilon_{1/2}=\tfrac{1}{2} \begin{pmatrix}
(I\pm J) & -(\omega_I^{-1} \mp \omega_J^{-1}) \\
(\omega_I \mp \omega_J) & -(I^*\pm J^*) \end{pmatrix}$, so that \ref{f:Gualtierimap} can be expressed as $\JJ_{1/2} = e^b\Upsilon_{1/2}e^{-b}$. Using this notation and the fact that $e^K$ and $e^b$ commute, it follows easily that
\begin{align*}
[\Phi_K(\JJ_1), \JJ_2] =&\ e^b[\Phi_K(\Upsilon_1), \Upsilon_2]e^{-b},\\
[\Phi_K(\JJ_2), \JJ_1)] =&\ e^b[\Phi_K(\Upsilon_2), \Upsilon_1]e^{-b}.
\end{align*}
Hence $[\Phi_K(\JJ_1), \JJ_2]=[\Phi_K(\JJ_2), \JJ_1)]$ reduces to the condition $[\Phi_K(\Upsilon_1), \Upsilon_2]=[\Phi_K(\Upsilon_2), \Upsilon_1]$. As a first step, we record the simplified forms of $\Phi_K(\Upsilon_{1/2})$ obtained through a direct computation:

\begin{equation*}\label{p:decomp}
\Phi_K(\Upsilon_{1/2})=\frac{1}{4}\left( \begin{matrix}
-(\omega_I^{-1}\mp\omega_J^{-1})K(I\pm J) & (\omega_I^{-1} \mp \omega_J^{-1})K(\omega_I^{-1}\mp\omega_J^{-1}) \\[6pt]
4K-(I^*\pm J^*)K(I\pm J)& (I^*\pm J^*)K(\omega_I^{-1}\mp \omega_J^{-1}) \\
\end{matrix} \right).
\end{equation*}
Further tedious computation yields
\begin{align*}
[\Phi_K(\JJ_1), \JJ_2] &= \tfrac{1}{2}\begin{pmatrix} (\omega_I^{-1}+\omega_J^{-1})K+g^{-1}K(I+J) & (\omega_I^{-1}-\omega_J^{-1})Kg^{-1}-g^{-1}K(\omega_I^{-1}-\omega_J^{-1}) \\[6pt]
 K(I-J) +(I^*-J^*)K&(I^*+J^*)Kg^{-1}-K(\omega_I^{-1}+\omega_J^{-1})\end{pmatrix}, \\[6pt]
 [\Phi_K(\JJ_2), \JJ_1] &= \tfrac{1}{2}\begin{pmatrix} (\omega_I^{-1}-\omega_J^{-1})K+g^{-1}K(I-J) & (\omega_I^{-1}+\omega_J^{-1})Kg^{-1}-g^{-1}K(\omega_I^{-1}+\omega_J^{-1}) \\[6pt]
 K(I+J) +(I^*+J^*)K&(I^*-J^*)Kg^{-1} -K(\omega_I^{-1}-\omega_J^{-1})\end{pmatrix}.
\end{align*}
By comparing each entry of the matrices above, we see equality holds if and only if $KJ = -J^*K$, as required.
\end{proof}

\subsection{Integrability Condition}

We next address the integrability condition.  Since our deformations should preserve integrability of each generalized complex structure $\JJ_i$, Proposition \ref{p:integrability} yields two partial integrability conditions which $K$ must satisfy.  We again emphasize that neither of these conditions alone will force $d K = 0$, while somewhat surprisingly the combination of the two conditions does.

\begin{prop} \label{p:GKint} Given $M$ a smooth manifold and $(\JJ_1, \JJ_2)$ a generalized K\"ahler structure, for $K \in \Lambda^2$ one has
\begin{align*}
\pi_{0,1}^{\JJ_i} \left(\iota_X\iota_YdK\right) =&\ 0 \quad \text{\ for all \ }  X, Y \in \pi_T(L_i), \quad i = 1,2.
\end{align*}
if and only if
\begin{align*}
d K = 0.
\end{align*}
\begin{proof} The sufficiency of $dK = 0$ is obvious, we prove it is necessary.  Note that for a pure covector $\xi$, one has $\JJ_{1/2}\xi =e^b\Upsilon_{1/2}\xi =-\tfrac{1}{2} \begin{pmatrix}(\omega_I^{-1}\mp\omega_J^{-1})\xi\\ b(\omega_I^{-1}\mp\omega_J^{-1})\xi + (I^*\pm J^*)\xi \end{pmatrix}$ and so 

\begin{align}
\label{projLbar}\pi_{0,1}^{1/2}(\xi)&=\begin{pmatrix}\frac{-\sqrt{-1}}{2}(\omega_I^{-1}\mp\omega_J^{-1})\xi\\ \xi-\frac{\sqrt{-1}}{2}\left(b(\omega_I^{-1}\mp\omega_J^{-1})\xi + (I^*\pm J^*)\xi \right) \end{pmatrix}.
\end{align}

Fix vectors $X, Y \in T^{1,0}_I$, and then choose lifts $X_+, Y_+$ to $C_+$, the $+1$-eigenspace of $\GG$.  Using the representation of $\JJ_i$ with respect to the $\pm 1$-eigenspace decomposition induced by $\GG$ (\cite{GualtieriThesis} Proposition 6.12), it follows that $X_+, Y_+ \in L_1 \cap L_2 = L_1^+$.  Now let $\xi = i_Y i_X d K$, and note that Proposition \ref{p:integrability} applied to both $\JJ_1$ and $\JJ_2$ implies that
\begin{align*}
\pi_{0,1}^{1/2}(\xi) =0.
\end{align*}
Comparing against equation (\ref{projLbar}) we obtain $(\omega_I^{-1}\mp \omega_J^{-1})(\xi)=0$. Therefore $\iota_Y\iota_XdK= 0$, for all $X,Y \in T^{1,0}_I$.  Since $K$ is real it follows that  $dK = 0$.
\end{proof}
\end{prop}

\begin{proof}[Proof of Theorem \ref{t:maindeformationthm}] Fix $(\JJ_1, \JJ_2)$ generalized K\"ahler and fix $(\JJ_1^t, \JJ_2^t)$ a one-parameter family as in the statement.  First let us assume conditions (1), (2), and (3) hold for this family.  Since $dK_t = 0$ for all $t$, it follows from Corollary \ref{c:integrability} that $(\JJ_1^t, \JJ_2^t)$ are integrable generalized complex structures.  Furthermore, using that $K_t \in \Lambda^{1,1}_{J^t}$ for all $t$, it follows from Lemma \ref{l:variation} and Proposition \ref{p:GKcompatibilityprop} that $[\JJ_1^t, \JJ_2^t] = 0$ for all $t$.  Since we have assumed the positivity of $\IP{- \JJ_1 \JJ_2 \cdot, \cdot}$ in condition (3), it follows that $(\JJ_1^t, \JJ_2^t)$ is generalized K\"ahler for all $t$.

Conversely, suppose $(\JJ_1^t, \JJ_2^t)$ defines a generalized K\"ahler structure for all $t$.  Condition (3) then holds by definition, and condition (1) holds by Lemma \ref{l:variation} and Proposition \ref{p:GKcompatibilityprop}.  As the structures $(\JJ_1^t, \JJ_2^t)$ are assumed integrable for all times $t$, their Nijenhuis tensors vanish for all $t$, and thus it follows from Proposition \ref{p:integrability} that
\begin{align*}
0 =&\ \pi_{0,1}^{\JJ_i} d K \left( \pi_T \pi_{1,0}^{\JJ_i} (\vec{x}), \pi_T \pi_{1,0}^{\JJ_i} (\vec{y}) , \cdot \right)
\end{align*}
for all $\vec{x}, \vec{y} \in T \oplus T^*$, all $t$, and $i=1,2$.  It then follows from Proposition \ref{p:GKint} that $dK = 0$, as required.
\end{proof}

With this characterization of canonical deformations in hand, we can now give the definition of generalized K\"ahler \emph{classes} which emerges naturally from Theorem \ref{t:maindeformationthm}.

\begin{defn} \label{d:GKclass} Let $M$ be a smooth manifold.  An \emph{exact canonical deformation} is a one-parameter family of generalized K\"ahler structures $(\JJ^t_1, \JJ^t_2)$ such that, for all $t$,
\begin{align*}
\dt \JJ_i^t = \Phi_{d a_t} \JJ_i^t,
\end{align*}
for some $a_t \in \Lambda^1$.  Given $\til{\JJ}_1, \til{\JJ}_2$ another generalized K\"ahler structure, we define an equivalence relation where
\begin{align*}
(\til{\JJ}_1, \til{\JJ}_2) \sim_{\mbox{\tiny{exact}}} (\JJ_1, \JJ_2)
\end{align*}
if and only if there exists an exact canonical deformation $(\JJ_1^t, \JJ_2^t)$, $t \in [0,1]$, such that $(\JJ_1^0, \JJ_2^0) = (\JJ_1, \JJ_2)$, $(\JJ_1^1, {\JJ}_2^1) = (\til{\JJ}_1, \til{\JJ}_2)$.  Furthermore, the \emph{generalized K\"ahler class of $(\JJ_1, \JJ_2)$} is
\begin{align*}
[(\JJ_1, \JJ_2)] = \{ (\til{\JJ}_1, \til{\JJ}_2) \mbox{ generalized K\"ahler } |\ (\til{\JJ}_1, \til{\JJ}_2) \sim_{\mbox{\tiny{exact}}} (\JJ_1, \JJ_2) \}.
\end{align*}
\end{defn}

\subsection{Induced variations}

In this section we derive the variation on the associated bihermitian data induced by a canonical deformation through an analysis of the Gualtieri map.

\begin{prop} \label{p:GKvariation} Let $(\JJ_1^t, \JJ_2^t)$ be a canonical family, and let $(g_t, b_t, I_t, J_t)$ denote the corresponding $1$-parameter family of bihermitian data.  Then
\begin{gather} \label{f:BHvariation}
\begin{split}
\dot{g} =-\tfrac{1}{2}[K, I],&\ \qquad \dot{b} =-\tfrac{1}{2}\{K, I\},\\
\dot{\gw}_I =-\tfrac{1}{2}[K,I]I, &\ \qquad \dot{\gw}_J =-\tfrac{1}{2}\{K, IJ\}, \\
\dot{I} =0, &\ \qquad \dot{J} =\tfrac{1}{2}[I,J]g^{-1}K.
\end{split}
\end{gather}
\begin{proof} We use the notation and computations of Proposition \ref{p:GKcompatibilityprop}.  In particular, we recall that 
\begin{align*}
e^{-b} \dot{\JJ}_{1/2} e^b = \frac{1}{4}\begin{pmatrix}-(\omega_I^{-1}\mp\omega_J^{-1})K(I\pm J)&(\omega_I^{-1}\mp\omega_J^{-1})K(\omega_I^{-1}\mp \omega_J^{-1})\\[4pt] 4K-(I^*\pm J^*)K(I\pm J)&(I^*\pm J^*)K(\omega_I^{-1}\mp\omega_J^{-1}) \end{pmatrix}.
\end{align*}
On the other hand, writing expression \ref{p:decomp} as $\JJ_{1/2} = \tfrac{1}{2}e^b\Upsilon_{1/2}e^{-b}$ and differentiating, using the fact that $e^b\dot{e}^b = \dot{e}^b = \dot{e}^be^b$, yields 
\begin{align}\label{J_i-derivative}
\begin{split}
e^{-b}\dot{\JJ}_{1/2}e^b =&\ \tfrac{1}{2} e^{-b} \left( \dot{e}^b\Upsilon_{1/2}e^{-b}+ e^b\dot{\Upsilon}_{1/2}e^{-b}- e^b \Upsilon_{1/2}\dot{e}^b \right) \\
=&\ \tfrac{1}{2}[\dot{e}^b, \Upsilon_{1/2}] + \tfrac{1}{2}\dot{\Upsilon}_{1/2}\\
=&\ \tfrac{1}{2}\begin{pmatrix}(\omega_I^{-1}\mp\omega_J^{-1})\dot{b}+(\dot{I}\pm\dot{J})& -(\dot{\omega}_I^{-1}\mp\dot{\omega}_J^{-1})\\[6pt] \{ \dot{b}, I\pm J\}+(\dot{\omega}_I\mp \dot{\omega}_J)& -\dot{b}(\omega_I^{-1}\mp \omega_J^{-1}) -(\dot{I}^*\pm \dot{J}^*) \end{pmatrix},
\end{split}
\end{align}
Then equating the appropriate expressions coming from above shows
\begin{align*}
e^{-b}(\dot{\JJ}_1 +\dot{\JJ}_2)e^b =&\ \begin{pmatrix} \dot{I} + \omega_I^{-1}\dot{b} & \omega_I^{-1}\dot{\omega}_I\omega_I^{-1}\\[4pt] \dot{\omega}_I + \{\dot{b}, I\} &-(\dot{I}^*+\dot{b}\omega_I^{-1}) \end{pmatrix}\\
=&\ \tfrac{1}{2}\begin{pmatrix}\omega_J^{-1}KJ-\omega_I^{-1}KI &\omega_I^{-1}K\omega_I^{-1} + \omega_J^{-1}K\omega_J^{-1}\\[4pt] 4K-(I^*KI +J^*KJ)&I^*K\omega_I^{-1}-J^*K\omega_J^{-1} \end{pmatrix},\\[8pt]
e^{-b}(\dot{\JJ}_1 -\dot{\JJ}_2)e^b =&\ \begin{pmatrix} \dot{J} - \omega_J^{-1}\dot{b} & -\omega_J^{-1}\dot{\omega}_J\omega_J^{-1}\\[4pt] \{\dot{b}, J\}-\dot{\omega}_J &\dot{b}\omega_J^{-1}-\dot{J}^* \end{pmatrix}\\
=&\ \tfrac{1}{2}\begin{pmatrix} \omega_J^{-1}KI-\omega_I^{-1}KJ &-(\omega_I^{-1}K\omega_J^{-1} + \omega_J^{-1}K\omega_I^{-1})\\[4pt] -(I^*KJ +J^*KI)&J^*K\omega_I^{-1}-I^*K\omega_J^{-1} \end{pmatrix}.
\end{align*}
Turning to the top right entries of each expression shows 
\begin{align*}
\dot{\omega}_I &= \tfrac{1}{2}(K +\omega_I\omega_J^{-1}K\omega_J^{-1}\omega_I) =\tfrac{1}{2}( K + I^*J^*KJI) = \tfrac{1}{2}(K + I^*KI) = -\tfrac{1}{2}[K, I]I,\\
\dot{\omega}_J&=\tfrac{1}{2}(\omega_J\omega_I^{-1}K + K\omega_I^{-1}\omega_J)=-\tfrac{1}{2}(J^*I^*K +KIJ)=-\tfrac{1}{2}\{K, IJ\}.
\end{align*}

For the remaining data we differentiate the generalized metric $\GG_t = -(\JJ_1)_t(\JJ_2)_t= \begin{pmatrix}A_t&g^{-1}_t\\ \delta_t& A_t^* \end{pmatrix}$, where  $g_t$ is the associated metric and $b_t = -g_tA_t$. It follows that 
\begin{equation*}
e^{-b}\dot{\GG}e^b =\begin{pmatrix}\dot{A}-g^{-1}\dot{g}g^{-1}b&-g^{-1}\dot{g}g^{-1}\\-b\dot{A} + \dot{A}^*b + \dot{\delta}-bg^{-1}\dot{g}g^{-1}&\dot{A}^*-bg^{-1}\dot{g}g^{-1} \end{pmatrix}
\end{equation*}
and in particular we will only be interested in the first row.  Focusing on the top row, we furthermore compute
\begin{align*}
e^{-b}\dot{\GG}e^b&=-(e^{-b}\dot{\JJ}_1e^b)(e^{-b}\JJ_2e^b) -(e^{-b}\JJ_1e^b)(e^{-b}\dot{\JJ}_2e^b)\\
&=-\tfrac{1}{8}\begin{pmatrix}-(\omega_I^{-1}-\omega_J^{-1})K(I+ J)&(\omega_I^{-1}-\omega_J^{-1})K(\omega_I^{-1}- \omega_J^{-1})\\[4pt] 4K-(I^*+ J^*)K(I+ J)&(I^*+ J^*)K(\omega_I^{-1}-\omega_J^{-1}) \end{pmatrix}\begin{pmatrix}I-J& -(\omega_I^{-1}+\omega_J^{-1})\\ \omega_I+\omega_J& -(I^*-J^*) \end{pmatrix}\\
&-\tfrac{1}{8}\begin{pmatrix}I+J& -(\omega_I^{-1}-\omega_J^{-1})\\ \omega_I-\omega_J& -(I^*+J^*) \end{pmatrix}\begin{pmatrix}-(\omega_I^{-1}+\omega_J^{-1})K(I- J)&(\omega_I^{-1}+\omega_J^{-1})K(\omega_I^{-1}+ \omega_J^{-1})\\[4pt] 4K-(I^*- J^*)K(I- J)&(I^*- J^*)K(\omega_I^{-1}+\omega_J^{-1}) \end{pmatrix}\\
&=\tfrac{1}{2}\begin{pmatrix} g^{-1}K(I-J) + (\omega_I^{-1}-\omega_J^{-1})K &-(\omega_I^{-1}-\omega_J^{-1})Kg^{-1}-g^{-1}K(\omega_I^{-1}+\omega_J^{-1})\\ * &*\end{pmatrix}\\
&= \tfrac{1}{2}\begin{pmatrix}g^{-1}\{K, I\} & -(\omega_I^{-1}-\omega_J^{-1})Kg^{-1}-g^{-1}K(\omega_I^{-1}+\omega_J^{-1})\\ *&* \end{pmatrix}.
\end{align*}
Thus
\begin{align*}
\dot{g} =&\ \tfrac{1}{2}g(\omega_I^{-1}-\omega_J^{-1})K + \tfrac{1}{2}K(\omega_I^{-1}+\omega_J^{-1})g\\
=&\ \tfrac{1}{2}(I^*K-KI) -\tfrac{1}{2}(KJ+J^*K)\\
=&\ -\tfrac{1}{2}[K,I],
\end{align*}
where we have used that $K \in \Lambda^{1,1}_J$.  Then differentiating the formulas $I/J = -\omega_{I/J}^{-1}g$ gives
\begin{align*}
\dot{I}&=\omega_I^{-1}\dot{\omega}_I\omega_I^{-1}g-\omega_I^{-1}\dot{g}\\
&=-\tfrac{1}{2}\omega_I^{-1}[K,I]I\omega_I^{-1}g+\tfrac{1}{2}\omega_I^{-1}[K,I]=0,\\
\dot{J}&=\omega_J^{-1}\dot{\omega}_J\omega_J^{-1}-\omega_J^{-1}\dot{g}\\
&=-\tfrac{1}{2}\omega_J^{-1}\{K,IJ\}\omega_J^{-1}+\tfrac{1}{2}\omega_J^{-1}[K,I]\\
&=-\tfrac{1}{2}\omega_J^{-1}KI-\tfrac{1}{2}g^{-1}I^*KJ + \tfrac{1}{2}\omega_J^{-1}KI-\tfrac{1}{2}\omega_J^{-1}I^*K\\
&= \tfrac{1}{2}g^{-1}I^*J^*K -\tfrac{1}{2}g^{-1}J^*I^*K= \tfrac{1}{2}[I, J]g^{-1}K.
\end{align*} 
Lastly, differentiating the formula $b= -gA$ yields 
\begin{align*}
\dot{b}&= -\dot{g}A- g\dot{A}\\
&= -\dot{g}A-g(g^{-1}\dot{g}g^{-1} b + \tfrac{1}{2}g^{-1}\{K, I\})\\
&=-\tfrac{1}{2}\{K, I\}.
\end{align*}
\end{proof}
\end{prop}

Using Proposition \ref{p:GKvariation} we give the proof of Corollary \ref{c:Poissoncor}.

\begin{proof}[Proof of Corollary \ref{c:Poissoncor}]
By Proposition \ref{f:BHvariation} we know that $\dot{I} = 0$, and moreover it follows that 
\begin{align*}
\dot{\sigma}&=\tfrac{1}{2}\left(I[I, J]g^{-1}K -[I, J]g^{-1}KI\right)g^{-1}+\tfrac{1}{2}[I, J]g^{-1}[K,I]g^{-1} \\
&=\tfrac{1}{2}I[I,J]g^{-1}Kg^{-1}-\tfrac{1}{2}[I,J]g^{-1}I^*Kg^{-1}\\
&=\tfrac{1}{2}\left(I[I,J]+[I,J]I\right)g^{-1}Kg^{-1}\\
&=0.
\end{align*}
\end{proof}

\subsection{Examples} \label{ss:examples}

\begin{ex} \label{e:KC} Given $(M^{2n}, g, J)$ a K\"ahler manifold, we interpret this as a generalized K\"ahler structure by setting $I = J$, and $b = 0$.  Suppose $(g_t, b_t, I_t, J_t)$ is a canonical family with this initial condition.  Initially we have $\gs = 0$, thus $\gs_t \equiv 0$ from Corollary \ref{c:Poissoncor}.  It follows that $[I, J] \equiv 0$ for all times, so from the equations of Proposition \ref{p:GKvariation} it follows that that $\dot{J} \equiv 0$ for all times, and so $J_t \equiv J = 0$.  In turn it follows easily that
\begin{align*}
\dot{\gw}_I = \dot{\gw}_J = K,
\end{align*}
in other words, the complex structures stay fixed and the K\"ahler forms change by $K$.  By the $\i\del\delb$-Lemma, an exact canonical deformation satisfies $K = da = d Jd u$ for some $u \in C^{\infty}(M)$.  Using the construction above and the normalization that $\i \del \delb = d I d$ we verify item (1) of Proposition \ref{p:examplesprop}, noting that positivity of the K\"ahler forms $(\gw_I)_t = \gw_I + t d J d u$ is equivalent to the positivity condition (3) in Theorem \ref{t:maindeformationthm}.
\end{ex}

\begin{ex} \label{e:CC} Suppose $(M^{2n}, g, b, I, J)$ is a generalized K\"ahler structure and $[I, J] \equiv 0$. These manifolds are characterized by a holomorphic splitting of the tangent bundle and simple examples are given by quotients of products, for instance on Hopf surfaces.  For further background on these structures see (\cite{ApostolovGualtieri}).  Setting $Q= - IJ$ we see that $Q^2 = 1$ and so we can split $T = T_+ \oplus T_-$ in terms of the $\pm1$-eigenbundles of $Q$.  Let $J_+ =\left. I\right|_{T_+}$ and $J_- =\left. I\right|_{T_-}$.  Since on $T_{\pm}$ we have $I = \pm J$, it follows $\left. J\right|_{T_+}= J_+$ and $\left. J\right|_{T_-}= - J_-$. Similarly, we define $\omega_{\pm} = \left.\omega_I\right|_{T_{\pm}}$.  Note that the K\"ahler form indeed block diagonalizes along $T_{\pm}$ since
\begin{align*}
\omega_I(X_+, IY_-) = g(IX_+, IY_-)= g(JX_+, JY_-) = g(-IX_+, IY_-) = -\omega_I(X_+, IY_-) = 0.
\end{align*}
Now suppose $(g_t, b_t, I_t, J_t)$ is a canonical family with this initial condition.  Arguing as above, since $\gs = 0$ for the given structure, it follows that $[I, J] \equiv 0$ for all times $t$, and hence $\dot{J} = 0$.  Thus along the variation we preserve the splitting induced by $Q$, and we can decompose $K = K_{++} + K_{+-} + K_{-+} + K_{--}$.  Tracing through the formulas in Proposition \ref{p:GKvariation} yields
\begin{align*}
\dot{g} =&\ \begin{pmatrix}J_+^*K_{++}&0\\0& J_-^*K_{--} \end{pmatrix},\qquad \dot{b}= \begin{pmatrix}0&-J_+^*K_{-+} \\ -J_-^*K_{+-}&0 \end{pmatrix},\\
\dot{\omega}_I =&\ \begin{pmatrix} K_{++}&0\\0&K_{--}\end{pmatrix}, \qquad \dot{\omega}_J =\begin{pmatrix} -K_{++}&0\\0&K_{--}\end{pmatrix}.
\end{align*}
A special case of this occurs when $K = d J d u$, yielding
\begin{align*}
\dot{\gw}_I =&\ \i \left( \del_+ \delb_+ - \del_- \delb_- \right) u,
\end{align*}
where $d = \del_+ + \del_- + \delb_+ + \delb_-$ is the fourfold splitting of $d$ induced by $Q$.  This construction verifies item (2) of Proposition \ref{p:examplesprop}, again noting that positivity of
\begin{align*}
(\gw_I)_t = \gw_I + t \i \left(\del_+ \delb_+ - \del_- \delb_- \right) u
\end{align*}
is equivalent to the positivity condition of item (3) in Theorem \ref{t:maindeformationthm}.
\end{ex}

\begin{ex} \label{e:NDcase} Suppose $(M^{2n}, g, b, I, J)$ is a generalized K\"ahler structure where the associated Poisson tensor $\gs$ is nondegenerate.  In this setting the endomorphism $[I,J]$ is invertible, and this in turn implies that $I \pm J$ are invertible.  We define the 2-forms $F_{\pm} = -2g(I \pm J)^{-1}$.  Moreover let $\Omega = \gs^{-1}$. It turns out that the three symplectic forms $F_{\pm}, \Omega$ completely determine the generalized K\"ahler structure in this case.   Direct computations show that the generalized complex structures can be expressed as
\begin{align*}
\JJ_1=e^{-4\Omega}\begin{pmatrix}0 & -F_-^{-1}\\F_-&0 \end{pmatrix}e^{4\Omega}, \quad \JJ_2 = \begin{pmatrix} 0  & -F_+^{-1}\\ F_+&0 \end{pmatrix},
\end{align*}
Let $K$ be an infinitessmal deformation of the generalized Kahler pair $(\JJ_1, \JJ_2)$.  Then with respect to the data $(F_+, F_-, \Omega)$, direct computations show that
\begin{align*}
\dot{F}_+ = \dot{F}_- = K, \qquad \dot{\Omega} = 0.
\end{align*}
As in the above examples we can choose $K_t = d J_t d u_t$.  We claim that for such a variation,
\begin{align*}
\dot{J} = L_{\gs d u} J.
\end{align*}
To show this we compute, using that $\gs = \Omega^{-1}$ and $d \Omega = 0$,
\begin{align*}
\left(L_{\gs d u} J \right) \Omega =&\ L_{\gs d u} (J \Omega) - J L_{\gs d u} \Omega = d J d u = K.
\end{align*}
Comparing against Proposition \ref{p:GKvariation} we see that
\begin{align*}
\dot{J} = \gs K = L_{\gs d u} J,
\end{align*}
as required.  It follows that $J_t = \phi_t^* J_0$, where $\phi$ is the one-parameter family of $\Omega$-Hamiltonian diffeomorphisms driven by $u_t$.  As discussed in \cite{ASNDGKCY}, the positivity of $- \Img \pi_{\Lambda_{I}^{1,1}} \Omega_J$ is equivalent to the positivity condition (3) of Theorem \ref{t:maindeformationthm}.  This verifies item (3) of Proposition \ref{p:examplesprop}.
\end{ex}

\begin{proof}[Proof of Proposition \ref{p:examplesprop}] The three deformations claimed in the proposition are described in the examples above.
\end{proof}

\section{Generalized K\"ahler-Ricci flow as canonical deformation} \label{s:GKRF}

In this section we establish Theorem \ref{t:GKflowthm}, namely that solutions to the generalized K\"ahler-Ricci flow are canonical deformations, driven by the Bismut Ricci curvature.  This generalizes and unifies various instances of this phenomenon which have previously been observed.  In particular, it is well known that K\"ahler-Ricci flow moves within the K\"ahler cone against a fixed complex structure, and so is a canonical deformation as in Example  \ref{e:KC}.  Next, comparing Example \ref{e:CC} against the curvature identities of (\cite{StreetsCGKFlow}), we see that the generalized K\"ahler-Ricci flow in the case $[I, J] = 0$ is a canonical deformation driven by $\rho_I$.  Also, we can compare the discussion in Example \ref{e:NDcase} with (\cite{ASNDGKCY}) to see the phenomenon holds in the nondegenerate case.  The
key point in establishing the general case is to show that the evolution of $J$ is indeed determined by $K = \rho_I$.  This requires a delicate curvature identity we build up below.  

\begin{lemma} \label{l:BismutRicci} (\cite{IvanovPapa}) Let $(M^{2n}, g, I)$ be a pluriclosed 
structure.  Then
\begin{gather} \label{f:BismutRicci}
 \begin{split}
  \rho_I(X,Y) =&\ - \Rc^B(X,IY) + \N^B_X \theta (IY)\\
  \rho^{1,1}_I(\cdot,I \cdot) =&\ \Rc^g - \tfrac{1}{4} H^2 - \tfrac{1}{2} L_{\theta^{\sharp}} g,\\
  \rho_I^{2,0 + 0,2}(X,Y) =&\ \tfrac{1}{2} \left( d^* H(IX, Y) + d^{\N} \theta(IX, Y) \right)\\
  =&\ \left(d I \theta \right)^{2,0 + 0,2}(X, Y).
  \end{split}
\end{gather}
\end{lemma}

\begin{prop} \label{p:sigchern2} Let $(M^{2n}, g, I, J)$ be a generalized K\"ahler structure.  Then
\begin{align*}
g \left( \left( L_{\theta_J^{\sharp} - \theta_I^{\sharp}} J \right) X, Y \right)=&\ \rho_I([I,J] X, Y).
\end{align*}
\begin{proof} We will use that for a Hermitian manifold $(M^{2n}, g, J)$  one has
\begin{align} \label{f:Liederiv}
g( (L_X J)Y, Z) =&\ g( (D_X J)Y - D_{JY} X + J D_Y X, Z).
\end{align}
First note, using that $\rho_I \in \Lambda^{1,1}_J$,
\begin{align*}
\rho_I([I,J]X, JY) &+ \rho_I([I,J]JY, X)\\
=&\ \rho_I( IJ X - JI X, JY) - \rho_I (X, IJJ Y - JIJ Y)\\
=&\ \rho_I(IJ X, JY) - \rho_I(I X, Y) + \rho_I(X, IY) - \rho_I(JX, IJY)\\
=&\ - \rho_I(JY, IJX) + \rho_I(Y, IX) + \rho_I(X, IY) - \rho_I(JX, IJY).
\end{align*}
Applying Lemma \ref{l:BismutRicci} we further compute
\begin{align*}
- \rho_I(JY, IJX)& + \rho_I(Y, IX) + \rho_I(X, IY) - \rho_I(JX, IJY)\\
=&\ - \Rc^I(JY, JX) + \N^I_{JY} \theta^I(JX) + \Rc^I(Y, X) - \N^I_Y \theta^I(X)\\
&\ + \Rc^I(X, Y) - \N^I_X \theta^I(Y) - \Rc^I(JX, JY) + \N^I_{JX} \theta^I(JY)\\
=&\ - \Rc^J(JY, JX) - \Rc^J(JX, JY) + \Rc^J(X,Y) + \Rc^J(Y,X)\\
&\ + \N^I_{JY} \theta^I(JX) - \N^I_Y \theta^I(X) - \N^I_X \theta^I(Y) + \N^I_{JX} \theta^I(JY)\\
=&\ \rho_J(JY, X) + \rho_J(JX, Y) + \rho_J(X,JY) + \rho_J(Y, JX)\\
&\ -\N^J_{JY} \theta^J(JX) - \N^J_{JX} \theta^J(JY) + \N^J_X \theta^J(Y) + \N^J_Y \theta^J(X) \\
&\ + \N^I_{JY} \theta^I(JX) - \N^I_Y \theta^I(X) - \N^I_X \theta^I(Y) + \N^I_{JX} \theta^I(JY)\\
=&\  D_X(\theta^J - \theta^I)(Y) + D_Y (\theta^J - \theta^I)(X) - D_{JX} \left(\theta^J - \theta^I \right) (JY) - D_{JY} \left(\theta^J - \theta^I \right)(JX)\\
=&\ g \left( \left(L_{\theta_J^{\sharp} - \theta_I^{\sharp}} J \right) X, JY \right) + g\left( \left(L_{\theta_J^{\sharp} - \theta_I^{\sharp}} J \right) JY, X \right).
\end{align*}
where the last line follows by comparing (\ref{f:Liederiv}).

To address the skew symmetric piece we first compute, again using that $\rho_I \in \Lambda^{1,1}_J$,
\begin{align*}
\rho_I & \left( [I,J] X, Y \right) - \rho_I \left( [I,J]Y, X \right)\\
=&\ \rho_I((IJ - JI) X, Y) + \rho_I(X,(IJ - JI)Y)\\
=&\ \rho_I(IJ X, Y) + \rho_I(IX, JY) + \rho_I(X, IJ Y) + \rho_I(JX, IY)\\
=&\ 2 \left(\left( \rho_I \right)^{2,0}(JX, IY) + \left( \rho_I \right)^{2,0}(IX, JY) \right).
\end{align*}
Now we compute using Lemma \ref{l:BismutRicci} and equation (\ref{f:Liederiv}),
\begin{align*}
2 & \left(\left( \rho_I \right)^{2,0}(JX, IY) + \left( \rho_I \right)^{2,0}(IX, JY) \right)\\
=&\ 2 \left(\left( - \rho_I \right)^{2,0}(IY,JX) + \left( \rho_I \right)^{2,0}(IX, JY) \right)\\
=&\ d^* H_I(Y, JX) - d^* H_I(X, JY) + d^{\N^I} \theta^I(Y, JX) - d^{\N^I} \theta^I (X, JY)\\
=&\ - d^* H_J (Y, JX) + d^* H_J(X, JY) + d^{\N^I} \theta^I(Y, JX) - d^{\N^I} \theta^I (X, JY)\\
=&\ 2 (\rho_J)^{2,0} (X,Y) - d^{\N^J} \theta^J(JX, Y) - 2 (\rho_J)^{2,0}(Y,X) + d^{\N^J} \theta^J(JY, X)\\
&\ + d^{\N^I} \theta^I(Y, JX) - d^{\N^I} \theta^I (X, JY)\\
=&\  d^{\N^J} \theta^J(JX, Y) - d^{\N^J} \theta^J(JY, X) + d^{\N^I} \theta^I(Y, JX) - d^{\N^I} \theta^I (X, JY)\\
=&\ g \left( \left( L_{\theta_J^{\sharp} - \theta_I^{\sharp}} J \right) X, Y \right) -  g \left( \left( L_{\theta_J^{\sharp} - \theta_I^{\sharp}} J \right) Y, X \right),
\end{align*}
as required.
\end{proof}
\end{prop}

\begin{proof}[Proof of Theorem \ref{t:GKflowthm}]
We first assume that equations (\ref{f:GKRFGG}) hold.  By Proposition \ref{p:GKvariation} with $K = \rho_I$ one notes that $I$ is fixed and one easily derives the evolution equation $\dot{\gw}_I = - (\rho_I)^{1,1}$.  By Proposition \ref{p:sigchern2} it follows that
\begin{align*}
\dot{J} = \tfrac{1}{2} [I,J] g^{-1} \rho_I = L_{\tfrac{1}{2} \left(\theta_J^{\sharp} - \theta_I^{\sharp}\right)} J.
\end{align*}
as required.  Thus we have verified the evolution equations of (\ref{f:GKRFBH}).  Assuming equations (\ref{f:GKRFBH}) and imposing $\dot{b} = - \{\rho_B^I, I\}$, we can apply the computations in Proposition \ref{p:GKvariation} and use Proposition \ref{p:sigchern2} to establish equations (\ref{f:GKRFGG}).
\end{proof}

\bibliographystyle{acm}

\end{document}